\documentclass[oneside,english]{amsart}
\usepackage[latin9]{inputenc}
\usepackage{mathrsfs}
\usepackage{amstext}
\usepackage{amsthm}
\usepackage{amssymb}
\usepackage{stackrel}

\makeatletter
\numberwithin{equation}{section}
\numberwithin{figure}{section}
\theoremstyle{plain}
\newtheorem{thm}{\protect\theoremname}
\theoremstyle{definition}
\newtheorem{problem}[thm]{\protect\problemname}
\theoremstyle{plain}
\newtheorem{lem}[thm]{\protect\lemmaname}
\theoremstyle{plain}
\newtheorem{prop}[thm]{\protect\propositionname}
\theoremstyle{plain}
\newtheorem{cor}[thm]{\protect\corollaryname}

\makeatother

\usepackage{babel}
\providecommand{\corollaryname}{Corollary}
\providecommand{\lemmaname}{Lemma}
\providecommand{\problemname}{Problem}
\providecommand{\propositionname}{Proposition}
\providecommand{\theoremname}{Theorem}

\begin{document}
\title{RICHNESS OF ARITHMATIC PROGRESSIONS IN COMMUTATIVE SEMIGROUP }
\author{ANINDA CHAKRABORTY AND SAYAN GOSWAMI}
\email{{\Large{}anindachakraborty2@gmail.com}}
\address{{\Large{}Government General Degree College at Chapra/ University of
Kalyani}}
\email{{\Large{}sayan92m@gmail.com}}
\address{{\Large{}Department of Mathematics, University of Kalyani}}
\thanks{{\Large{}The second author of the paper is supported by UGC-JRF fellowship.}}
\begin{abstract}
{\Large{}Furstenberg and Glasner proved that for an arbitrary $k\in\mathbb{N}$,
any piecewise syndetic set contains $k-$term arithmetic progressions
and such collection is also piecewise syndetic in $\mathbb{Z}.$ They
used algebraic structure of $\beta\mathbb{N}$. The above result was
extended for arbitrary semigroups by Bergelson and Hindman, again
using the structure of Stone-\v{C}ech compactification of general
semigroup. Beiglboeck provided an elementary proof of the above result
and asked whether the combinatorial argument in his proof can be enhanced
in a way which makes it applicable to a more abstract setting. In
a recent work the second author of this paper and S.Jana provided
an affirmative answer to Beiglboeck's question for countable commutative
semigroup. In this work we will extend the result of Beiglboeck in
different type of settings.}{\Large\par}
\end{abstract}

\maketitle

\section{{\Large{}Introduction}}

{\Large{}A subset $S$ of $\mathbb{\ensuremath{Z}}$ is called syndetic
if there exists $r\in\mathbb{N}$ such that $\bigcup_{i=1}^{r}(S-i)=\mathbb{Z}$
and it is called thick if it contains arbitrary long intervals in
it. Sets which can be expressed as intersection of thick and syndetic
sets are called piecewise syndetic sets. All these notions have natural
generalization for arbitrary semigroups. }{\Large\par}

{\Large{}One of the famous Ramsey theoretic result is so called Van
der Waerden\textquoteright s Theorem \cite{key-10} which states that
atleast one cell of any partition $\{C_{1},C_{2},\ldots,C_{r}\}$
of $\mathbb{N}$ contains arithmetic progressions of arbitrary length.
Since arithmetic progressions are invariant under shifts, it follows
that every piecewise syndetic set contains arbitrarily long arithmetic
progressions. The following theorem is due to Van der Waerden \cite{key-10}}{\Large\par}
\begin{thm}
{\Large{}Given any $r,l\in\mathbb{N}$, there exists $N(r,l)\in\mathbb{N}$,
such that for any $r-$partition of $[1,N]$, atleast one of the partition
contains an $l-$length arithmetic progression.}{\Large\par}
\end{thm}

{\Large{}Furstenberg and E. Glasner in \cite{key-5} algebraically
and Beiglboeck in \cite{key-1} combinatorially proved that if $S$
is a piecewise syndetic subset of $\mathbb{Z}$ and $l\in\mathbb{N}$
then the set of all $l$ length progressions contained in $S$ is
also large. The statement is the following:}{\Large\par}
\begin{thm}
{\Large{}\label{Thm 2}Let $k\in\mathbb{N}$ and assume that $S\subseteq\mathbb{Z}$
is piecewise syndetic. Then $\{(a,d)\,:\,\{a,a+d,\ldots,a+kd\}\subset S\}$
is piecewise syndetic in $\mathbb{Z}^{2}$.}{\Large\par}
\end{thm}

{\Large{}In the recent work \cite[Theorem 6]{key-9}, authors have
extended the technique of Beigelboeck in general commutative semigroup
and proved the following:}{\Large\par}
\begin{thm}
{\Large{}\label{Thm 3} Let $(S,+)$ be a commutative semigroup and
$F$ be any finite subset of $S$. Then for any piecewise syndetic
set $M\subseteq S$, the collection $\{(a,n)\in S\times\mathbb{N}:\,a+nF\subset M\}$
is piecewise syndetic in $(S\times\mathbb{N},+)$. }{\Large\par}
\end{thm}

{\Large{}The above theorem involves more general Gallai type progression.
But parallely the following problem comes from theorem \ref{Thm 2} }{\Large\par}
\begin{problem}
{\Large{}\label{prob 4} Let $S$ be a countable commutative semigroup
and $A$ be any piecewise syndetic subset of $S$. Then for any $l\in\mathbb{N}$,
is it possible that 
\[
\{(s,t)\in S\times S:\:\{s,s+t,s+2t,\ldots,s+dt\}\subseteq A\}
\]
is piecewise syndetic in $S\times S$.}{\Large\par}
\end{problem}

{\Large{}In this moment we are unable to give complete answer to this
question but we have given proof of a weak version of the theorem
for countable commutative semigroup. We will also give an answer of
\ref{prob 4} for some special kind of semigroups including divisible
semigroups.}{\Large\par}

\section{{\Large{}Proof of our results}}

{\Large{}The following lemma was proved in} {\Large{}\cite[Lemma 4.6( I')]{key-2}
for general semigroup by using algebraic structure of Stone-\v{C}ech
compactification of arbitrary semigroup and in \cite[lemma 8]{key-9}
for commutative semigroup by combinatorially.}{\Large\par}
\begin{lem}
{\Large{}\label{Lemma 5} Let $(S,+)$ and $(T,+)$ be commutative
semigroups, $\varphi:S\rightarrow T$ be a homomorphism and $A\subseteq S$.
Then if $A$ is piecewise syndetic in $S$ and $\varphi(S)$ is piecewise
syndetic in $T$, implies that $\varphi(A)$ is piecewise syndetic
in $T.$}{\Large\par}
\end{lem}

{\Large{}Now we need the following useful lemma,}{\Large\par}
\begin{lem}
{\Large{}\label{Lemma 6}If $M\subseteq S\times S$ is piecewise syndetic,
then for any $c\in S$ and $a\in\mathbb{N}$,}{\Large\par}

{\Large{}
\[
\{(s+at+c,t):(s,t)\in M\}
\]
}{\Large\par}

{\Large{}is piecewise syndetic in $S\times S$.}{\Large\par}
\end{lem}

\begin{proof}
{\Large{}Let $c\in S$ and consider the following homomorphism $\psi_{c}:S\times S\longrightarrow S\times S$
by $\psi_{c}(s,t)=(s+c,t)$. This map preserves piecewise syndeticity.}{\Large\par}

{\Large{}As $A\subseteq S\times S$ is piecewise syndetic set, there
exists a finite subset say $E_{1}=\{(a_{1},b_{1}),(a_{2},b_{2}),\ldots,(a_{r},b_{r})\}$
of $S\times S$ such that $\stackrel[i=1]{r}{\cup}(-(a_{i},b_{i})+A)$
is thick and since $\stackrel[i=1]{r}{\cup}(-(a_{i},b_{i})+A)\subseteq\stackrel[i=1]{r}{\cup}(-(a_{i}+c,b_{i})+\psi_{c}(A))$,
the set $\stackrel[i=1]{r}{\cup}(-(a_{i}+c,b_{i})+\psi_{c}(A))$ is
thick. So we have $\psi_{c}(A)$ is piecewise syndetic.}{\Large\par}

{\Large{}Now, for any $a\in\mathbb{N}$, the semigroup homomorphism
defined by $\varphi_{a}:S\times S\longrightarrow S\times S$ by $\varphi_{a}(s,t)=(s+at,t)$,
$\varphi_{a}(S\times S)$ is thick in $S\times S$ and hence piecewise
syndetic. So from \ref{Lemma 5}, this map preserves piecewise syndeticity.}{\Large\par}
\end{proof}
{\Large{}The following is a weaker version of problem \ref{prob 4}.}{\Large\par}
\begin{thm}
{\Large{}\label{Theorem 7} Let $S$ be any countable commutative
semigroup and $A$ be piecewise syndetic in $S$. Then for $l\in\mathbb{N}$,
there exists $d\in\mathbb{N}$ such that 
\[
\{(s,t)\in S\times S:\{s,s+dt,\ldots,s+ldt\}\subseteq A\}
\]
is piecewise syndetic in $S\times S$.}{\Large\par}
\end{thm}

\begin{proof}
{\Large{}Since $A$ is piecewise syndetic in $S$, then there exists
a finite subset $E$ of $S$, such that $\cup_{t\in E}-t+A$ is thick
in $S$.}{\Large\par}

{\Large{}Let $|E|=r$ and say $E=\{c_{1},c_{2},\ldots,c_{r}\}$ and
let $M(r,l)=N$ be the Van der Waerden number. }{\Large\par}

{\Large{}The set of all possible $l-$length arithmetic progressions
in $[1,N]$ is finite as $[1,N]$ is finite. et $H=\{h_{1},h_{2},\ldots,h_{n}\}$
be the set of such progressions with $|H|=n$ (say).}{\Large\par}

{\Large{}Then, for any $(s_{1},t_{1})\in S\times S$, if the set $\{s_{1}+t_{1},s_{1}+2t_{1},\ldots,s_{1}+Nt_{1}\}$
will be partitioned into $r$ cells, one of the partition will contain
a $l$-length arithmetic progression.}{\Large\par}

{\Large{}Consider the set $B=\{(s,t)\in S\times S:\:s+[1,N]t\subseteq\cup_{t\in E}-t+A\}.$
It is easy to verify that $B$ is thick in $S\times S$.}{\Large\par}

{\Large{}Of course for any finite set $K=\{(s_{1},t_{1}),(s_{2},t_{2}),\ldots,(s_{m},t_{m})\}$
and $t\in S$ a translation of the set $\{s_{i}+[1,N](t_{i}+t)\}_{i=1}^{m}$
by an element $a\in S$ (say) will be contained in $\cup_{t\in E}-t+A$.
This gives the required translation of $K$ by $(a,t)\in S\times S$.}{\Large\par}

{\Large{}Now the set $B$ can be $|H\times E|-$colored in a way that
we will give an element $(s,t)$ of $B$ the color $(i,j)\in[1,n]\times[1,r]$
if for the least $i$, the set $\{s+h_{i}t\}\subseteq-c_{j}+A$ with
the least $j\in[1,r]$. }{\Large\par}

{\Large{}Then, as we have partitioned the thick set $B$, one of them
will be piecewise syndetic. Let the set 
\[
Q=\{(s,t)\in B:\:\{s+at,s+at+dt,\ldots,s+at+ldt\}\subseteq-c_{j}+A\text{ for some }j\in[1,r]\}
\]
 is piecewise syndetic in $S\times S.$}{\Large\par}

{\Large{}Now, the set $\tilde{Q}=\{(s+at+c_{j},t):(s,t)\in Q\}$ is
piecewise syndetic by lemma \ref{Lemma 6} and this proves the theorem.}{\Large\par}
\end{proof}
{\Large{}Since for commutative semigroup $G$ it is not necessary
that for any $g\in G$, $gG$ is a piecewise syndetic in $G$,e.g.
take any $n\in\mathbb{N}$, $n\mathbb{Z}[x]$ isn't piecewise syndetic
in $\mathbb{Z}[x]$.}{\Large\par}

{\Large{}Now we are taking $\mathscr{A}$ as the collection of all
those countable commutative semigroups $(S,+)$ for which $dS=\{dx:\:x\in S\}\subseteq S$
is piecewise syndetic in $S$. Clearly $\mathscr{A}$ includes all
the divisible semigroups such as $(\mathbb{Q},+),(\mathbb{Q^{+}},+),(\mathbb{Q}/\mathbb{Z},+)$
etc. and others like $\mathbb{Z}$,$\mathbb{N}$,$\mathbb{Z}[i]$
etc. We will say a semigroup $(S,+)$ is a semigroup of class $\mathscr{A}$
if $S\in\mathscr{A}$.}{\Large\par}
\begin{lem}
{\Large{}\label{Lemma 8} Let $S$ be a countable commutative semigroup
of class $\mathscr{A}$ and $M\subseteq S\times S$ is piecewise syndetic
then for any $d\in\mathbb{N}$,}{\Large\par}

{\Large{}
\[
\{(s,dt):(s,t)\in M\}
\]
}{\Large\par}

{\Large{}is piecewise syndetic in $S\times S$.}{\Large\par}
\end{lem}

\begin{proof}
{\Large{}Let $d\in\mathbb{N}$ and define $\chi_{d}:S\times S\longrightarrow S\times S$
as $\chi_{d}(s,t)=(s,dt)$. Then $\chi_{d}$ preserves piecewise syndeticity
as from \ref{Lemma 5} and the fact that $\chi_{d}(S\times S)$ is
piecewise syndetic in $S\times S$.}{\Large\par}
\end{proof}
{\Large{}So we have the following result:}{\Large\par}
\begin{prop}
{\Large{}\label{Prop 9} Let $S$ be a countable commutative semigroup
of class $\mathscr{A}$ and $A$ be piecewise syndetic in $S$. Then
for $l\in\mathbb{N}$, then
\[
\{(s,t)\in S\times S:\{s,s+t,s+2t,\ldots,s+dt\}\subseteq A\}
\]
is piecewise syndetic in $S\times S$.}{\Large\par}
\end{prop}

{\Large{}In this moment we are unable to derive the above proposition
for general commutative semigroup which will give an affirmtive answer
of problem \ref{prob 4} and leave the question open.}{\Large\par}

\section{{\Large{}Applications}}

{\Large{}The set $AP^{l+1}=\{(a,a+b,a+2b,\ldots,a+lb):a,b\in S\}$
is a commutative subsemigroup of $S^{l+1}$. Using a result deduced
in \cite[Theorem 3.7 (a)]{key-3} it is easy to see that for any piecewise
syndetic set $A\subseteq S,\,A^{l+1}\cap AP^{l+1}$ is piecewise syndetic
in $AP^{l+1}$. Now as a consequence of proposition \ref{Prop 9}
we will derive this result not for all but for a large class of semigroups
in the following.}{\Large\par}
\begin{cor}
{\Large{}\label{Corollary 10} Let $(S,+)\in\mathscr{A}$ be a commutative
semigroup then for any piecewise syndetic set $M\subseteq S$, $M^{l+1}\cap AP^{l+1}$
is piecewise syndetic in $AP^{l+1}$.}{\Large\par}
\end{cor}

\begin{proof}
{\Large{}Let us take a surjective homomorphism $\varphi:S\times S\rightarrow AP^{l+1}$
by, $\varphi(a,b)=(a,a+b,a+2b,\ldots,a+lb)$.}{\Large\par}

{\Large{}Then from lemma \ref{Lemma 5} the map $\varphi$ preserves
the piecewise syndeticity. }{\Large\par}

{\Large{}Let $B=\{(s,t)\in S\times S:\{s,s+t,s+2t,\ldots,s+dt\}\subseteq A\}$
and from proposition \ref{Prop 9} $\varphi(B)$ is piecewise syndetic
in $AP^{l+1}$.}{\Large\par}

{\Large{}Now clearly, $\varphi(B)\subseteq M^{l}\cap AP^{l+1}$ and
from lemma \ref{Prop 9} we get our required result.}{\Large\par}

{\Large{}This proves the claim.}{\Large\par}
\end{proof}
{\Large{}Now we will give a combinatorial proof of proposition \ref{Prop 9}
replacing the condition of piecewise syndeticity by Quasi-central
set which is another notion of largeness and is very close to the
famous central set.}{\Large\par}

{\Large{}A quasi-central set is genarally defined in terms of algebraic
structure of $\beta\mathbb{N}$. But it has an combinatorial characterisation
which will be needed for our purpose, stated below.}{\Large\par}
\begin{thm}
{\Large{}\label{Th 11} \cite[Theorem 3.7]{key-11}For a countable
semigroup $(S,.)$, $A\subseteq S$ is said to be Quasi-central iff
there is a decreasing sequence $\langle C_{n}\rangle_{n=1}^{\infty}$
of subsets of $A$ such that,}{\Large\par}

{\Large{}$(1)$ \label{prop 1} for each $n\in\mathbb{N}$ and each
$x\in C_{n}$, there exists $m\in\mathbb{N}$ with $C_{m}\subseteq x^{-1}C_{n}$
and}{\Large\par}

{\Large{}$(2)$ \label{prop 2} $C_{n}$ is piecewise syndetic $\forall n\in\mathbb{N}$.}{\Large\par}
\end{thm}

{\Large{}The following lemma is essential for our result:}{\Large\par}
\begin{lem}
{\Large{}\label{lemma 12} The notion of quasi-central is preserved
under surjective semigroup homorphism}{\Large\par}
\end{lem}

\begin{proof}
{\Large{}Let $\varphi:S_{1}\longrightarrow S_{2}$ be a surjective
semigroup homomorphism. Let $A$ be quasi-central in $S_{1}$ and
then the following holds as in property 1 in theorem \ref{Th 11}.}{\Large\par}

{\Large{}
\[
A\supseteq A_{1}\supseteq A_{2}\supseteq\ldots\supseteq A_{n}\supseteq\ldots
\]
Now in $S_{2}$ consider the following sequence,}{\Large\par}

{\Large{}
\[
\varphi(A)\supseteq\varphi(A_{1})\supseteq\varphi(A_{2})\supseteq\ldots\supseteq\varphi(A_{n})\supseteq\ldots
\]
and due to surjectivity of $\varphi$, $\varphi(A)$ and $\varphi(A_{i})\text{ for }i\in\mathbb{N}$
are piecewise syndetic. }{\Large\par}

{\Large{}Choose $y\in\varphi(A_{m})$ for some $m\in\mathbb{N}$ and
then there exists some $x\in A_{m}$ such that $\varphi(x)=y$ and
consider the set $-y+\varphi(A_{m})$ . Now as $-x+A_{m}\supseteq A_{n}$
for some $n$, we have for any $z\in A_{n}$, $x+z\in A_{m}$ and
then $y+\varphi(z)\in\varphi(A_{m})$ and so $\varphi(z)\in-y+\varphi(A_{m})$. }{\Large\par}

{\Large{}Hence $-y+\varphi(A_{m})\supseteq\varphi(A_{n})$ and as
all $y,m,n$ are chosen arbitrarily, we have the required proof.}{\Large\par}
\end{proof}
{\Large{}Now we will deduce proposition \ref{Prop 9} for quasi-central
sets:}{\Large\par}
\begin{thm}
{\Large{}Let $(S,+)$ be a countable commutative semigroup of class
$\mathscr{A}$. Then for any quasi-central $M\subseteq S$ the collection
$\{(a,b):\,\{a,a+b,a+2b,\ldots,a+lb\}\subset M\}$ is quasi-central
in $(S\times S,+)$.}{\Large\par}
\end{thm}

\begin{proof}
{\Large{}As, $M$ is quasi-central, theorem \ref{Th 11} guarantees
that there exists a decreasing sequence $\{A_{n}\}_{n\in\mathbb{N}}$
of piecewise syndetic subsets of $S$, such that property 1 of theorem
\ref{Th 11} is satiesfied.}{\Large\par}

{\Large{}As $A_{n}$ is piecewise syndetic $\forall n\in\mathbb{N}$
in the following sequence,}{\Large\par}

{\Large{}\label{eqn 1}
\[
1.\quad M\supseteq A_{1}\supseteq A_{2}\supseteq\ldots\supseteq A_{n}\supseteq\ldots
\]
}{\Large\par}

{\Large{}The set $B=\{(a,b):\,\{a,a+b,a+2b,\ldots,a+lb\}\subset M\}$
is piecewise syndetic in $S\times S$ from proposition \ref{Prop 9}.}{\Large\par}

{\Large{}And for $i\in\mathbb{N},$ $B_{i}=\{(a,b)\in S\times S:\,\{a,a+b,a+2b,\ldots,a+lb\}\subset A_{i}\}\neq\phi$
is piecewise syndetic $\forall i\in\mathbb{N},$ proposition \ref{Prop 9}.}{\Large\par}

{\Large{}Consider,\label{eqn 2}}{\Large\par}

{\Large{}
\[
2.\quad B\supseteq B_{1}\supseteq B_{2}\supseteq\ldots\supseteq B_{n}\supseteq\ldots
\]
}{\Large\par}

{\Large{}Now choose $n\in\mathbb{N}$ and $(a,b)\in B_{n}$, then
$\{a,a+b,a+2b,\ldots,a+lb\}\subset A_{n}$. Then by property 1 we
have
\begin{equation}
A_{N}\subseteq\stackrel[i=0]{l}{\bigcap}(-(a+ib)+A_{n}).
\end{equation}
}{\Large\par}

{\Large{}As for any $(a_{1},b_{1})\in B_{N}$ we have 
\[
\{a_{1},a_{1}+b_{1},a_{1}+2b_{1},\ldots,a_{1}+lb_{1}\}\subseteq A_{N}\subseteq\stackrel[i=0]{l}{\bigcap}(-(a+ib)+A_{n})
\]
 and $(a_{1}+a)+i(b_{1}+b)\in A_{n}\forall i\in\{0,1,2,\ldots,l\}$.
Therefore $(a_{1},b_{1})\in-(a,b)+B_{n}$. Which implies $B_{N}\subseteq-(a,b)+B_{n}$,
showing the property 1 of theorem \ref{Th 11}.}{\Large\par}

{\Large{}This proves the theorem.}{\Large\par}
\end{proof}
{\Large{}The following is an extension of corollary \ref{Corollary 10}.}{\Large\par}
\begin{cor}
{\Large{}\label{Corollary 14} Let $(S,+)$ be a commutative semigroup
of class $\mathscr{A}.$ Then for any quasi-central set $M\subseteq S$,
$M^{l+1}\cap AP^{l+1}$ is quasi-central in $AP^{l+1}$.}{\Large\par}
\end{cor}

\begin{proof}
{\Large{}Let us take a surjective homomorphism $\varphi:S\times S\rightarrow AP^{l+1}$
by, $\varphi(a,b)=(a,a+b,a+2b,\ldots,a+lb)$.}{\Large\par}

{\Large{}As $M$ is quasi central, from property 1 of theorem \ref{Th 11},
it satiesfies equation \ref{eqn 1}.}{\Large\par}

{\Large{}Now from equation \ref{eqn 2},}{\Large\par}

{\Large{}Where the set $B$ and $B_{i}\ (f\text{ or }i\in\mathbb{N})$
are from previous theorem and it was shown that $B$ is quasi central.}{\Large\par}

{\Large{}Now clearly, $\varphi(B_{i})\subseteq M^{l}\cap AP^{l+1}$
for each $i\in\mathbb{N}$ and from lemma \ref{lemma 12} we get our
required result.}{\Large\par}

{\Large{}This proves the claim.}{\Large\par}
\end{proof}
{\Large{}However there are other different type of notion of largeness
such as $IP-sets,Central\,sets,J,sets,C\,sets,D\,sets$ all of those
have combinatorial characterizations described in \cite{key-7} but
we don't know if it is possible to give an affirmative answer of the
problem \ref{prob 4}. }{\Large\par}

\end{document}